\documentclass[11pt]{amsart}
\usepackage[colorlinks, citecolor=blue, dvipdfm, pagebackref]{hyperref}

\usepackage{extarrows}

\newtheorem{theorem}{Theorem}[section]
\newtheorem{lemma}[theorem]{Lemma}

\theoremstyle{definition}
\newtheorem{definition}[theorem]{Definition}

\newtheorem{example}[theorem]{Example}

\newtheorem{proposition}[theorem]{Proposition}
\newtheorem{corollary}[theorem]{Corollary}
\newtheorem{remark}[theorem]{Remark}
\newtheorem{conjecture}[theorem]{Conjecture}

\theoremstyle{remark}

\newcommand{\be}{\begin{equation}}
\newcommand{\ee}{\end{equation}}
\newcommand{\hooklongrightarrow}{\lhook\joinrel\longrightarrow}
\numberwithin{equation}{section}

\begin{document}
\title{The rigidity on the second fundamental form of projective manifolds}
\author{Ping Li}
\address{School of Mathematical Sciences, Tongji University, Shanghai 200092, China}

\email{pingli@tongji.edu.cn\\
pinglimath@gmail.com}

\thanks{The author was partially supported by the National
Natural Science Foundation of China (Grant No. 11722109).}




\begin{abstract}
Let $M$ be a complex $n$-dimensional projective manifold in $\mathbb{P}^{n+r}$ endowed with the Fubini-Study metric of constant holomorphic sectional curvature $1$, $\sigma$ its second fundamental form, and $\underline{|\sigma|}^2$ the mean value of the squared length of $\sigma$ on $M$. We derive a formula for $\underline{|\sigma|}^2$ and
classify them when $\underline{|\sigma|}^2\leq2n$. We present several applications to these results. The first application is to confirm a conjecture of Loi and Zedda, which characterizes the linear subspace and the quadric in terms of the $L^2$-norm of $\sigma$. The second application is to improve a result of Cheng solving an old conjecture of Oguie from pointwise case to mean case. The third application is to give an optimal second gap value on $\underline{|\sigma|}^2$, which can be viewed as a complex analog to those on minimal submanifolds in the unit spheres.
\end{abstract}

\maketitle
\section{Introduction}\label{section1}
To put our motivation and results into perspective, we start by recalling some results related to the second fundamental form of submanifolds in the unit spheres and standard complex projective spaces.

Let $M$ be an $n$-dimensional compact minimal submanifold in the unit sphere $S^{n+r}$ with second fundamental form $\sigma$. In a seminal paper \cite{Si}, Simons discovered a gap phenomenon for the squared length of $\sigma$: if $|\sigma|^2\leq n/(2-\frac{1}{r})$ everywhere on $M$, then either $\sigma\equiv0$, i.e., $M$ is totally geodesic, or $|\sigma|^2\equiv n/(2-\frac{1}{r})$. Soon afterwards Chern, do Carmo and Kobayashi (\cite{CDK}) showed that Simon's result is sharp and classified the equality case. The case of $r=1$ was also independently obtained by Lawson (\cite{La}). Chern also proposed to study the subsequent gaps for $|\sigma|^2$ (\cite[p. 42]{Ch}, \cite[p. 75]{CDK}) and this was collected by Yau in his famous problem section (\cite[p. 693]{Ya}). For minimal hypersurfaces in $S^{n+1}$, i.e., the case of codimension $r=1$, the investigation of the second gap for $|\sigma|^2$ was initiated by Peng and Terng (\cite{PT1}, \cite{PT2}) and, after some works (\cite{WX}, \cite{Zh}), their estimate was eventually improved by Ding and Xin (\cite{DX}) to show that for each $n$, there exists a positive constant $\delta(n)$ depending only on $n$ such that if $|\sigma|^2> n$, then $|\sigma|^2\geq n+\delta(n)$. Due to their method employed the gap $\delta(n)$ is by no means optimal and the conjectured optimal gap is $n$, i.e., if $|\sigma|^2> n$, then $|\sigma|^2\geq 2n,$ which remains open. It is known that the case of $2n$ can be achieved by Cartan's isoparametric minimal hypersurfaces in $S^{n+1}$.

It is well-known that any (holomorphically immersed) complex submanifold of a K\"{a}hler manifold is minimal and so is natural to exploit similar gap phenomena for complex submanifolds in complex projective spaces endowed with the Fubini-Study metric of constant holomorphic sectional curvature $1$. This was initiated by Ogiue in \cite{Og1} and shortly afterwards there appeared a series of related papers on this topic, whose main results have been summarized by Ogiue in his influential semi-expository paper \cite{Og2}. Ogiue also posed several open problems in \cite[p. 112]{Og2} to characterize the totally geodesic submanifolds in $\mathbb{P}^{n+r}$ in terms of various curvature pinching conditions, and some of them have been answered or partially answered (\cite{Che}, \cite{Ro}, \cite{RV}, \cite{Li1}, \cite{Li2}), to the author's best knowledge. In particular, Ogiue conjectured that (\cite[p. 112, Problem 3]{Og2}) if $|\sigma|^2<n$ everywhere on a complete complex $n$-dimensional holomorphically immersed
 submanifold $M$ in $\mathbb{P}^{n+r}$, then it must be totally geodesic. This was resolved by Cheng and Liao (\cite{Che}, \cite{Li1}) when $M$ is compact.

The previously-mentioned results indicate that, in the class of compact minimal (resp. complex) submanifolds in a sphere (resp. complex projective space), the totally geodesic submanifolds are isolated. 
This motivates Gromov to conjecture in \cite{Gr} that every \emph{smooth} immersed map
of a compact smooth manifold into a compact quotient of the complex hyperbolic space, whose second fundamental form is small, is homotopic to a totally geodesic submanifold. Here the term ``small" is only qualitative and has not been precisely formulated. Besson, Courtois and Gallot answered it in \cite{BGG} in the holomorphic case in terms of the $L^2$ and $L^{2n}$ norms of the second fundamental form. To be precise, they showed that (\cite[p. 151]{BGG}) a holomorphic immersion of a compact complex $n$-dimensional K\"{a}hler manifold into a compact quotient of the complex hyperbolic space, whose $|\sigma|^{2}_{L^2}$ and $|\sigma|^{2}_{L^{2n}}$ are smaller than a positive constant depending only on $n$, is totally geodesic.

Various characterizations of $\mathbb{P}^n$ and the hyperquadric in $\mathbb{P}^{n+1}$ have a long history since the pioneer works of Hirzebruch-Kodaira, Brieskorn, and Kobayashi-Ochiai (\cite{HK}, \cite{Br}, \cite{KO}). Motivated by the result in \cite{BGG}, Loi and Zedda also addressed in \cite{LZ} the problem of finding the \emph{optimal} constant $c(n)$ depending only on $n$ to ensure that, if $|\sigma|^2_{L^2}<c(n)$ for an $n$-dimensional projective manifold in $\mathbb{P}^{n+r}$, then $M$ is totally geodesic, and formulated the following (\cite[p. 69]{LZ})
\begin{conjecture}[Loi-Zedda]\label{conj}
Let $M$ be an $n$-dimensional projective manifold in $\mathbb{P}^{n+r}$ with the induced metric. If \be\label{inequality1}
|\sigma|^2_{L^2}:=\int_M|\sigma|^2\ast1<2n
\cdot\text{Vol}(\mathbb{P}^n),\ee
 where $\ast1$ is the volume form and $\text{Vol}(\mathbb{P}^n)$ the volume of the standard $n$-dimensional $\mathbb{P}^n$ in $\mathbb{P}^{n+r}$, then $M$ is isomorphic to the $n$-dimensional hyperplane $\mathbb{P}^n$. And the equality holds if and only if $M$ is isomorphic to the complex quadric
\be\label{quadric}Q^n:=\Big\{[z^0:\cdots:z^{n+1}:
\underbrace{0:\cdots:0}_{r-1}]
\in~\mathbb{P}^{n+r}~\big|~(z^0)^2+\cdots+
(z^{n+1})^2=0\Big\}.\nonumber\ee
\end{conjecture}
\begin{remark}
Loi and Zedda verified Conjecture \ref{conj} in the cases of $n=1$, $M$ is a complete intersection and $|\sigma|^2$ is constant (\cite[Theorem 3]{LZ}).
\end{remark}

The key idea in Simons's paper \cite{Si} is to calculate the Laplacian of $|\sigma|^2$ and estimate some terms involved to produce the desired results. This idea was more or less inherited by most papers related to the second fundamental form of complex submanifolds in complex projective spaces (e.g., those summarized in \cite{Og2}), but with one exception in \cite{Che}, where some arguments are of algeo-geometric nature. This motivates us to apply deeper results in algebraic geometry to treat this kind of problems, and the main tools employed in our paper are some classification results in the adjunction theory of algebraic geometry, mainly due to Fujita (\cite{Fu}).

The rest of this paper is organized as follows. The main results and several of their corollaries are stated in Section \ref{section2}. In Section \ref{section3} we briefly recall the notions of sectional genus and $\Delta$-genus in algebraic geometry and some classification results in terms of them, mainly due to Fujita. Then Sections \ref{section4}, \ref{section5} and \ref{section4.3} are devoted to the proofs of various results described in Section \ref{section2}.

\section*{Acknowledgements}
The author wishes to thank Professor Yuan-Long Xin for drawing his attention to this kind of rigidity results in minimal submanifold theory.

\section{Main results}\label{section2}
Our first observation is the following formula for $\underline{|\sigma|}^2$.
\begin{theorem}\label{firstresult}
Let $M\overset{i}{\hooklongrightarrow}\mathbb{P}^{n+r}$ be an $n$-dimensional projective manifold in $\mathbb{P}^{n+r}$ with
the induced metric, $\sigma$ the second fundamental form of $M$, and $L$ the hyperplane section bundle on $M$, i.e., $L=i^{\ast}\big(\mathcal{O}_{\mathbb{P}^{n+r}}(1)\big)$. Then
\be\label{fundamentalformformula}\underline{|\sigma|}^2=2n
\big[1+\frac{g(L)-1}{d(M)}\big],\ee
where $d(M)$ is the degree of $M$ in $\mathbb{P}^{n+r}$ and $g(L)\in\mathbb{Z}$ the sectional genus of $L$.
\end{theorem}
\begin{remark}
The notion of \emph{sectional genus} of a line bundle comes from algebraic geometry and
more related details shall be explained in Section \ref{section3}. By definition the degree of $M$ is nothing but $L^n$: $d(M)=L^n$.
\end{remark}

In order to state the classification result, we recall in the next example some special polarized manifolds, which are called \emph{rational normal scrolls} in the literature. We refer the reader to \cite[p. 5-7]{EH1} or \cite[\S 9.1.1]{EH2} for more details.

\begin{example}[Rational normal scroll]
Let $\varepsilon$ be a direct sum of $n$ line bundles of positive degrees over $\mathbb{P}^1$, i.e.,
$$\varepsilon=\bigoplus_{i=1}^n\mathcal{O}_{\mathbb{P}^1}(a_i),\qquad a_i\in\mathbb{Z}_{>0}.$$
Write $\mathbb{P}(\varepsilon)$ for the projectivization of $\varepsilon$ and let $\mathcal{O}_{\mathbb{P}(\varepsilon)}(1)$ be the tautological line bundle on $\mathbb{P}(\varepsilon)$, which is very ample under the conditions that all the degrees $a_i>0$. Denote for our later convenience
\be\label{notatinforscorll}
\big(S(a_1,\ldots,a_n),\mathcal{O}(1)\big):=
\big(\mathbb{P}(\varepsilon),
\mathcal{O}_{\mathbb{P}(\varepsilon)}(1)\big)\ee
and call this polarized pair a \emph{rational normal scroll}.
\end{example}

With this notion in hand
we are able to classify $M$ where $\underline{|\sigma|}^2\leq2n$. Since the classification  for $\underline{|\sigma|}^2<2n$ can be made more explicitly and the related applications also focus on it,  in the sequel we separately describe the two cases of ``$\underline{|\sigma|}^2<2n$" and ``$\underline{|\sigma|}^2=2n$".
\begin{theorem}\label{secondresult}
Let $(M,L)$ be as in Theorem \ref{firstresult} and  assume that $\underline{|\sigma|}^2<2n.$ Then the pair $(M,L)$ is isomorphic to
\begin{enumerate}
\item
$(\mathbb{P}^n,\mathcal{O}_{\mathbb{P}^n}(1))$, in which case $\sigma\equiv0$ and $d(M)=1$;

\item
$(Q^n,\mathcal{O}_{Q^n}(1))$, in which case $|\sigma|^2\equiv n$, $d(M)=2$ and the codimension $r\geq 1$;

\item
the Veronese surface $\big(\mathbb{P}^2,\mathcal{O}_{\mathbb{P}^2}(2)\big)$, in which case $\underline{|\sigma|}^2=3$, $d(M)=4$ and the codimension $r\geq 3$;

or

\item
one of the rational normal scrolls $\big(S(a_1,\ldots,a_n),\mathcal{O}(1)\big)$ in the notation of (\ref{notatinforscorll}) with all $a_i>0$, in which case
\begin{eqnarray}\label{norm of scroll}
\left\{ \begin{array}{ll}
\underline{|\sigma|}^2=2n(1-\frac{1}
{\sum_{i=1}^na_i})\\
~\\
d(M)=\sum_{i=1}^na_i,\\
\end{array} \right.
\end{eqnarray}
and the codimension $r\geq-1+\sum_{i=1}^na_i$.
\end{enumerate}

Furthermore, in the above cases, the lower bounds of the codimension $r$ are exactly realized by the Kodaira maps of these very ample line bundles $L$.
\end{theorem}
\begin{remark}\label{remark2.6}
The four classes above are disjoint with exactly one exception $\big(Q^2,\mathcal{O}_{Q^2}(1)\big)$, which is also isomorphic to the rational normal scroll $\big(S(1,1), \mathcal{O}(1)\big)$.
\end{remark}

Observe from (\ref{fundamentalformformula}) that the pairs $(M,L)$ where $\underline{|\sigma|}^2=2n$ correspond exactly to those with sectional genera $1$, which have been classified by Fujita (\cite[p. 107]{Fu}). So consequently we have
\begin{proposition}
Let $(M,L)$ be as above. If $\underline{|\sigma|}^2=2n$, then the pair $(M,L)$ is either
\begin{enumerate}
\item
a del Pezzo manifold, i.e., the canonical line bundle $K_M=(1-n)L$;

or

\item
a scroll over an elliptic curve, i.e., $M$ is the projectivization of a vector bundle on an elliptic curve and $L$ its tautological line bundle.
\end{enumerate}
\end{proposition}

Next we present several applications to Theorems \ref{firstresult} and \ref{secondresult}, whose detailed proofs shall be given in Section \ref{section4.3}.

The first one is a confirmation of Loi and Zedda's Conjecture \ref{conj}:
\begin{corollary}\label{thirdresult}
Conjecture \ref{conj} is true.
\end{corollary}

As mentioned above, Cheng and Liao showed that (\cite{Che}, \cite{Li1}), if for an $n$-dimensional compact (holomorphically immersed) complex submanifold whose $|\sigma|^2<n$ \emph{everywhere} on $M$, then $M$ is totally geodesic, which confirmed a conjecture of Oguie (\cite[p. 112, Prob. 3]{Og2}) in the compact case. As an immediate consequence of Theorem \ref{secondresult} as well as Remark \ref{remark2.6}, for embedded case this result can be improved from the pointwise case to the mean case.
\begin{corollary}\label{fourthresult}
For an $n$-dimensional projective manifold $M$ in $\mathbb{P}^{n+r}$ with the induced metric, if $\underline{|\sigma|}^2<n$ (resp. $\underline{|\sigma|}^2=n$), then $M$ is isomorphic to $\mathbb{P}^n$ (resp. the quadric $Q^n$).
\end{corollary}

Another by-product of this classification result is the \emph{optimal} second gap value on $|\sigma|^2$, which is a solution to the complex analog of the problem proposed by Chern discussed in the Introduction.
\begin{corollary}\label{fifthresult}
For an $n$-dimensional projective manifold $M$ in $\mathbb{P}^{n+r}$ with the induced metric and $n\geq3$, if $\underline{|\sigma|}^2>n$, then $\underline{|\sigma|}^2\geq 2n-2$. In particular, if $|\sigma|^2\in(n,2n-2]$ everywhere on $M$, then $|\sigma|^2\equiv 2n-2$. They are optimal as the case $|\sigma|^2\equiv 2n-2$ can be achieved by the rational normal scroll $S(\underbrace{1,\ldots,1}_{n})$.
\end{corollary}

\section{Preliminaries}\label{section3}
We briefly recall in this section some related notation and results in algebraic geometry, mainly for our later purpose. For more details on these materials we refer the reader to \cite{Fu} and \cite[\S 3]{BS}.

Throughout this section we work over $\mathbb{C}$, the field of complex numbers.

Let $M$ be an $n$-dimensional smooth projective variety, $L$ a line bundle on it, and $K_M$ its canonical line bundle. By applying the Hirzebruch-Riemann-Roch formula to the holomorphic Euler characteristic $\chi(M,tL)$ ($t\in\mathbb{Z}$) and considering some special coefficients in front of $t^i$, it turns out that (cf. \cite[p. 25-26]{Fu}) the integer $$\big(K_M+(n-1)L\big)\cdot L^{n-1}$$
is \emph{even}.
With this fact in mind, the following two closely related notions were introduced by Fujita (cf. \cite[p. 26]{Fu}), who, in a series of papers, successfully described the structure of the pair $(M,L)$ for ample $L$ when they are small enough. These results have been summarized in his book \cite{Fu}.
\begin{definition}
Let $M$ be an $n$-dimensional smooth projective variety and $L$ a line bundle on it. The \emph{sectional genus} of $L$, $g(L)$, is defined by
\be\label{sectional genus}g(L):=\frac{\big(K_M+(n-1)L\big)\cdot L^{n-1}}{2}+1.\ee
The \emph{$\Delta$-genus} of $L$, $\Delta(L)$, is defined by
\be\label{delta genus}\Delta(L):=n+L^n-\text{dim}_{\mathbb{C}}H^0(M,L).\ee
Here as usual denote by $H^0(M,L)$ the complex vector space consisting of holomorphic sections of $L$.
\end{definition}

A fundamental result related to $g(L)$ and $\Delta(L)$ is the following result due to Fujita (\cite{Fu0}, \cite[p. 107, p. 35]{Fu}).

\begin{theorem}\label{Fujita1}
The sectional genus $g(L)\geq0$ if $L$ is ample, and the equality holds if and only if the $\Delta$-genus $\Delta(L)=0$. In the latter case $L$ is necessarily very ample.
\end{theorem}

The following classification result of $\Delta$-genera zero for ample $L$ is due to Fujita (\cite[p. 41]{Fu}).
\begin{theorem}\label{Fujita2}
Let $M$ be an $n$-dimensional smooth projective variety and $L$ an ample line bundle on it. Suppose that $\Delta(L)=0$. Then the pair $(M,L)$ is isomorphic to
\begin{enumerate}
\item
$(\mathbb{P}^n,\mathcal{O}_{\mathbb{P}^n}(1))$,

\item
$(Q^n,\mathcal{O}_{Q^n}(1))$,

\item
the Veronese surface $\big(\mathbb{P}^2,\mathcal{O}_{\mathbb{P}^2}(2)\big)$,

or

\item
the rational normal scroll $\big(S(a_1,\ldots,a_n),\mathcal{O}(1)\big)$, in the notation of (\ref{notatinforscorll}), with all integers $a_i>0$.
\end{enumerate}
\end{theorem}
\begin{remark}
It turns out that the manifolds listed in Theorem \ref{Fujita2} coincide with smooth projective varieties of \emph{minimal degrees} (cf. \cite[Thm 1]{EH1}), a fact that will be needed in our proof of Theorem \ref{secondresult}.
\end{remark}

\section{Proof of Theorem \ref{firstresult}}\label{section4}
Let $(M,g,J)\overset{i}{\hooklongrightarrow}(\mathbb{P}^{n+r},g_{0},J_{0})$
be an $n$-dimensional projective manifold in $\mathbb{P}^{n+r}$. Here $\mathbb{P}^{n+r}$ is endowed with the standard complex structure $J_0$ and the Fubini-Study metric $g_0$ of constant holomorphic sectional curvature $1$, and $M$ is endowed with the induced metric $g=i^{\ast}(g_0)$. The associated (normalized) K\"{a}hler form $\omega_0$ of $g_0$ under the homogeneous coordinate $[z^0:z^1:\cdots:z^n]$ is defined by
$$\omega_0:=
\frac{\sqrt{-1}}{2\pi}\partial\bar{\partial}\log(
\sum_{i=0}^n|z^i|^2).$$
With this normalized coefficient we have (cf. \cite[p. 165]{Zhe})
\begin{eqnarray}
\left\{ \begin{array}{ll}
\int_{\mathbb{P}^{n+r}}\omega_0^{n+r}=1\\
~\\
c_1\big(\mathcal{O}_{\mathbb{P}^{n+r}}(1)\big)
=[\omega_0]\in H^{1,1}(\mathbb{P}^{n+r};\mathbb{Z}).\\
\end{array} \right.\nonumber
\end{eqnarray}

Therefore $\omega:=i^{\ast}(\omega_0)$ is the associated  normalized K\"{a}hler form of $g$ such that
\begin{eqnarray}\label{1.5}
\left\{ \begin{array}{ll}
c_1(L)=[\omega]\in H^{1,1}(M;\mathbb{Z})\\
~\\
\int_M\omega^n=L^n=d(M)\\
\end{array} \right.
\end{eqnarray}
as $L=i^{\ast}\big(\mathcal{O}_{\mathbb{P}^{n+r}}(1)\big)$.

Let
$$\text{Ric}(g):=-\frac{\sqrt{-1}}{2\pi}\partial\bar{\partial}
\log\det{(g)}$$
be the (normalized) Ricci form of $g$, which represents the first Chern class of $M$ and thus the anti-canonical line bundle: \be\label{2.5}[\text{Ric}(g)]=c_1(K_M^{-1}).\ee

Denote by $S_g$ the scalar curvature function of $g$ on $M$ and it is a well-known fact that (cf. \cite[p. 60]{Sz})
\be\label{3}S_g\cdot\omega^n=n\cdot\text{Ric}(g)\wedge\omega^{n-1}.\ee
Another basic fact is that $S_g$ is related to the squared length of the second fundamental form $\sigma$ by
\be\label{4}|\sigma|^2=n(n+1)-S_g,\ee
which can be proved via the Gauss equation (cf. \cite[p. 77]{Og2}).

With these facts understood, we can proceed to prove Theorem \ref{firstresult}.
\begin{proof}
\be\begin{split}
\underline{|\sigma|}^2-2n
&=\frac{\int_M|\sigma|^2\omega^n}{\int_M\omega^n}-2n\\
&=n(n-1)-\frac{\int_MS_g\cdot\omega^n}{\int_M\omega^n}\qquad\big(\text{by
(\ref{4})}\big)\\
&=n\Big[\frac{\int_M\big((n-1)\omega-\text{Ric}(g)\big)
\wedge\omega^{n-1}}{\int_M\omega^n}\Big]\qquad\big(\text{by
(\ref{3})}\big)\\
&=n\frac{\big((n-1)L+K_M\big)\cdot L^{n-1}}{d(M)}\qquad\big(\text{by
(\ref{1.5}) and (\ref{2.5})}\big)\\
&=2n\Big[\frac{g(L)-1}{d(M)}\Big].\qquad\big(\text{by
(\ref{sectional genus})}\big)
\end{split}
\nonumber\ee
This yields the desired formula (\ref{fundamentalformformula}) and thus completes the proof of Theorem \ref{firstresult}.
\end{proof}

\section{Proof of Theorem \ref{secondresult}}\label{section5}
The proof shall be divided into three lemmas.

\begin{lemma}The pair $(M,L)$ in question must be isomorphic to one of the four cases described in Theorem \ref{secondresult}.
\end{lemma}
\begin{proof}
The condition is that $\underline{|\sigma|}^2<2n$. Combining this with (\ref{fundamentalformformula}) implies that $g(L)\leq 0$. However, in our situation $L$ is ample (indeed very ample) and so Theorem \ref{Fujita1} tells us that $g(L)\geq0$. Therefore the only possibility is that $g(L)=0$, which, again by Theorem \ref{Fujita1}, is equivalent to $\Delta(L)=0$. This enables us to apply Fujita's classification result, Theorem \ref{Fujita2}, to conclude that the pair $(M,L)$ in question must be isomorphic to one of the four cases claimed in Theorem \ref{secondresult}.
\end{proof}

\begin{lemma}The claims in the four cases on the second fundamental form $\sigma$ and the degree $d(M)$ are true.
\end{lemma}
\begin{proof}
First note that the formula (\ref{fundamentalformformula}) reduces to
\be\label{newfundamentalformformula}
\underline{|\sigma|}^2=2n
\big[1-\frac{1}{d(M)}\big]\ee
as $g(L)=0$ discussed in the above lemma.

Case $(1)$ is clear.

For Case $(2)$, $d(M)=\big[\mathcal{O}_{Q^n}(1)\big]^n=2$ and so (\ref{newfundamentalformformula}) implies that $\underline{|\sigma|}^2\equiv n$. However, the standard fact is that the scalar curvature in this case is $S_g\equiv n^2$ (cf. \cite[p. 82]{Og2}) and so $|\sigma|^2\equiv n$ via (\ref{4}).

Case $(3)$ is clear.

For Case $(4)$, the degree satisfies (\cite[p. 7]{EH1}) $$d\big(S(a_1,\ldots,a_n)\big)=\big[\mathcal{O}(1)\big]^n=\sum_{i=1}^na_i$$
and so (\ref{norm of scroll}) follows from (\ref{newfundamentalformformula}).
\end{proof}

\begin{lemma}The lower bounds of the codimension $r$ are sharp and exactly attained by the Kodaira maps of these very ample line bundles.
\end{lemma}
\begin{proof}
Let $(M,L)$ be any pair in Theorem \ref{secondresult} and $r_m$ the desired \emph{minimal} codimension.

This means that there exists an embedding $M\overset{i}{\hooklongrightarrow}\mathbb{P}^{n+r_m}$ with $i^{\ast}\big(\mathcal{O}_{\mathbb{P}^{n+r_m}}(1)\big)=L$
and $M$ is not contained in any hyperplane of $\mathbb{P}^{n+r_m}$.

A general fact tells us that in this case the codimension $r_m\leq L^n-1$ (\cite[Prop. 0]{EH1}). However, in our four cases the pairs $(M,L)$ exactly satisfy $r_m=L^n-1$ and are called smooth projective varieties of \emph{minimal degree} (\cite[Thm. 1]{EH1}), a fact traced back to del Pezzo and Bertini, and a new and modern treatment was presented in \cite{EH1}.

It suffices to show that the codimension of the Kodaira map induced by the very ample line bundle $L$ is exactly $L^n-1$. Indeed, the Kodaira map of $(M,L)$ is of the following form (\cite[p. 176]{GH}):
\be M\hooklongrightarrow
\mathbb{P}\big(H^0(M,L)^{\ast}\big)
\cong\mathbb{P}^{\text{dim}_{\mathbb{C}}H^0(M,L)-1}.\nonumber\ee
Note that for these $L$ we have $\Delta(L)=0$ and so (\ref{delta genus}) tells us that
$$\text{dim}_{\mathbb{C}}H^0(M,L)-1=n+L^n-1$$
and hence the codimension is exactly $L^n-1$. This completes this lemma and hence the whole proof of Theorem \ref{secondresult}.
\end{proof}

\section{Proof of Corollaries \ref{thirdresult} and \ref{fifthresult}}\label{section4.3}
\subsection{Proof of Corollary \ref{thirdresult}}
It suffices to show that the degree of $M$ in question is $1$ or $2$ respectively.

\begin{proof}
The inequality (\ref{inequality1}) is equivalent to
\be\label{inequality1-equiv}\underline{|\sigma|}^2<\frac{2n}
{d(M)}\ee
as (cf. \cite[Lemma 5]{LZ})
\begin{eqnarray}
\left\{ \begin{array}{ll}
|\sigma|^2_{L^2}=\underline{|\sigma|}^2\cdot\text{Vol}(M)\\
~\\
\text{Vol}(M)=d(M)\cdot\text{Vol}(\mathbb{P}^n).\\
\end{array} \right.\nonumber
\end{eqnarray}

Assume that the inequality (\ref{inequality1-equiv}) holds. This, together with (\ref{fundamentalformformula}), yields
\be\label{1}d(M)+g(L)<2.\ee
Again by Theorem \ref{Fujita1} the ampleness of $L$ implies that $g(L)\geq0$ and so the only solution to (\ref{1}) is $$\big(d(M),g(L)\big)=(1,0)$$
and so $M$ is isomorphic to the $n$-dimensional hyperplane $\mathbb{P}^n$.

If the equality case in (\ref{inequality1-equiv}) holds, then $d(M)+g(L)=2$,
whose only solution is $$(d(M),g(L))=(2,0),$$ from which the result follows.
\end{proof}

\subsection{Proof of Corollary \ref{fifthresult}}
Note that the degree of the rational normal scrolls in Theorem \ref{secondresult} is
$$d\big(S(a_1,\ldots,a_n)\big)=\sum_{i=1}^na_i\geq n,$$
with the equality achieved by $S(\underbrace{1,\ldots,1}_{n})$.
Accordingly their $\underline{|\sigma|}^2$, via (\ref{norm of scroll}), satisfies
$$\underline{\big|\sigma\big(S(a_1,\ldots,a_n)\big)\big|}^2\geq 2n-2,$$
also with the equality achieved by $S(1,\ldots,1)$.

It suffices to show that $$\big|\sigma\big(S(1,\ldots,1)\big)\big|^2\equiv 2n-2.$$ Indeed in this case the embedding of $S(1,\ldots,1)$ into $\mathbb{P}^{2n-1}$ induced by the Kodaira map of $L$ is the famous \emph{Segre embedding} (\cite[p. 52-53]{EH2}):
$$S(1,\ldots,1)\cong\mathbb{P}^{n-1}
\times\mathbb{P}^1\hooklongrightarrow\mathbb{P}^{2n-1}.$$ Furthermore the induced metric on $\mathbb{P}^{n-1}\times\mathbb{P}^1$ from $\mathbb{P}^{2n-1}$ is isometric to the product metric of the Fubini-Study metrics of constant holomorphic sectional curvature $1$ (\cite[p. 401]{CR}), say $(g_1,g_2)$. Recall that the (constant) scalar curvature of the Fubini-Study metric of constant holomorphic curvature $1$ of $\mathbb{P}^n$ is $n(n+1)$. So the scalar curvature of the product metric is
$$S_{(g_1,g_2)}
=S_{g_{1}}+S_{g_{2}}=(n-1)n+2=n^2-n+2$$
and thus by (\ref{4}) we have $|\sigma|^2\equiv 2n-2$. This completes the proof of Corollary \ref{fifthresult}.

\end{document}